\documentclass{amsart}
\usepackage{amssymb}
\usepackage{latexsym}
\usepackage{amsmath}
\usepackage{euscript}

   \def\sH{{\mathfrak H}}   
      \def\sL{{\mathfrak L}}

\def\sfi{{\text{\rm{\textsf i}}}}

      \def\dC{{\mathbb C}}

      \def\dR{{\mathbb R}}

      \def\cL{{\EuScript L}}

\def\bm\chi{\mbox{\boldmath$\chi$}}

\def\min{{\rm min\,}}

\def\ker{{\rm ker\,}}
\def\ran{{\rm ran\,}}
\def\cran{{\rm \overline{ran}\,}}
\def\dom{{\rm dom\,}}
\def\mul{{\rm mul\,}}

\def\cdom{{\rm \overline{dom}\,}}

\def\dim{{\rm dim\,}}

\let\xker=\ker \def\ker{{\xker\,}}

\def\cmr{{\dC \setminus \dR}}

\def\mul{{\text{\rm mul\,}}}

\def\senki{{\lbrack\negthinspace [\bot ]\negthinspace\rbrack}}
\def\senki+{{\lbrack\negthinspace [+] \negthinspace\rbrack}}

\DeclareMathOperator{\hoplus}{\, \widehat \oplus \,}

\newtheorem{theorem}{Theorem}[section]
\newtheorem{proposition}[theorem]{Proposition}
\newtheorem{corollary}[theorem]{Corollary}
\newtheorem{lemma}[theorem]{Lemma}
\theoremstyle{definition}

\newtheorem{definition}[theorem]{Definition}

\numberwithin{equation}{section}

\begin{document}

\title[Antitonicity for selfadjoint relations]
{Antitonicity of the inverse for selfadjoint matrices, operators,
and relations}

\author[J.~Behrndt]{Jussi~Behrndt}
\author[S.~Hassi]{Seppo~Hassi}
\author[H.S.V.~de Snoo]{Henk~de~Snoo}
\author[H.L.~Wietsma]{Rudi~Wietsma}

\address{Institut f\"ur Numerische Mathematik\\
Technische Universit\"at Graz \\
Steyrergasse 30\\
8010 Graz \\
Austria}
\email{behrndt@tugraz.at}

\address{Department of Mathematics and Statistics\\
University of Vaasa\\
P.O. Box 700, FI-65101 Vaasa\\
Finland}
\email{sha@uwasa.fi}

\address{Johann Bernoulli Institute for Mathematics and Computer Science\\
University of Groningen \\
P.O. Box 407, 9700 AK Groningen \\
Nederland} \email{desnoo@math.rug.nl}

\address{Department of Mathematics and Statistics\\
University of Vaasa\\
P.O. Box 700, FI-65101 Vaasa\\
Finland}
\email{rwietsma@uwasa.fi}




\thanks{
This research was supported by the grants from the Academy of
Finland (project 139102) and the German Academic Exchange Service
(DAAD); PPP Finland project 50740090. The third author would like to thank the
Deutsche Forschungsgemeinschaft (DFG) for the Mercator visiting
professorship at the Technische Universit\"at Berlin.
}

\begin{abstract}
Let $H_1$ and $H_2$ be selfadjoint operators or relations
(multivalued operators) acting on a separable Hilbert space and
assume that the inequality $H_1 \leq H_2$ holds. Then the validity
of the inequalities $-H_1^{-1} \le -H_2^{-1}$ and $H_2^{-1} \le
H_1^{-1}$ is characterized in terms of the inertia of $H_1$ and
$H_2$. Such results are known for matrices and boundedly invertible
operators. In the present paper those results are extended to
selfadjoint, in general unbounded, not necessarily boundedly
invertible, operators and, more generally, for selfadjoint relations
in separable Hilbert spaces.
\end{abstract}

\maketitle

\section{Introduction}

Let $H_1$ and $H_2$ be selfadjoint matrices, operators, or relations
(multivalued operators) in a separable Hilbert space, which is not
necessarily finite-dimensional. This paper is concerned with a
question which goes back to K.~L\"owner: what  are the implications
of the inequality $H_1 \le H_2$ for the inverses of $H_1$ and $H_2$;
cf. \cite{A,L34}.

Here specific conditions are investigated under which the implication
\begin{equation}\label{antieq}
H_1 \le H_2 \quad \Rightarrow \quad H_2^{-1} \le H_1^{-1}
\end{equation}
is true. In the literature such results are often formulated as
\textit{antitonicity results}, see e.g. \cite{BNS,HaNo,N,S}. Of
course, the above implication does not hold in general; a simple
counterexample is $H_1=-I$ and $H_2=I$. In the finite-dimensional
setting necessary and sufficient conditions for the implication in
\eqref{antieq} to hold are given by the following antitonicity
theorem, see \cite{BNS,N}. Recall that the inertia of the
selfadjoint matrix $H_i$, $i=1,2$, is the ordered triplet,
$\sfi(H_i) = \{\sfi_i^+,\sfi_i^-,\sfi_i^0\}$, of the numbers of
positive, negative, and zero eigenvalues of $H_i$.

\begin{theorem}\label{antith}
Let $H_1$ and $H_2$ be invertible selfadjoint matrices in $\dC^n$
and assume that $H_1 \leq H_2$. Then
\[ H_2^{-1} \leq H_1^{-1} \quad \textrm{if and only if} \quad {\sf i}(H_1) = {\sf i}(H_2).\]
\end{theorem}

The condition that the matrices $H_1$ and $H_2$ are invertible means
that $\sfi_1^0=\sfi_2^0=0$; hence the condition ${\sf i}(H_1) = {\sf
i}(H_2)$ in Theorem \ref{antith} is equivalent to
$\sfi_1^-=\sfi_2^-$ and to $\sfi_1^+=\sfi_2^+$. If in Theorem
\ref{antith} the matrices $H_1$ and $H_2$ are not invertible, then
the inverses $H_1^{-1}$ and $H_2^{-1}$ still exist in the sense of
linear relations (multivalued mappings). With this interpretation
Theorem \ref{antith} can be generalized to obtain the following two
results, which are new and applicable already in the
finite-dimensional setting (cf. \cite{BHSW3,BHSW2}).

\begin{theorem}\label{intro2}
Let $H_1$ and $H_2$ be selfadjoint relations in $\dC^n$ and
assume that $H_1 \leq H_2$. Then
\[
 H_2^{-1} \leq H_1^{-1} \quad \textrm{if and only if} \quad
 {\sf i}_1^- = {\sf i}_2^-.
 \]
\end{theorem}

\begin{theorem}\label{intro1}
Let $H_1$ and $H_2$ be selfadjoint relations in $\dC^n$ and
assume that $H_1 \leq H_2$. Then
\[
 -H_1^{-1} \leq -H_2^{-1} \quad \textrm{if and only if} \quad
 {\sf i}_1^- +{\sf i}_1^0 = {\sf i}_2^- +{\sf i}_2^0.
 \]
\end{theorem}

Clearly, when the selfadjoint relations $H_1$ and $H_2$ are
invertible matrices, then Theorem~\ref{intro2} and \ref{intro1}
coincide with Theorem~\ref{antith}. However, in the case of
non-invertible matrices $H_1$ and $H_2$ the above statements are new
extensions of Theorem~\ref{antith}. Note that, since $H_1^{-1}$ and
$H_2^{-1}$ are selfadjoint relations, the condition $ -H_1^{-1} \leq
-H_2^{-1}$ is in general different from the condition $H_2^{-1} \leq
H_1^{-1}$.

From either of the above theorems also other previously known
antitonicity results in the matrix literature can be derived as
special cases. For example the main antitonicity result for the
Moore-Penrose inverse $H^+$ of a selfadjoint matrix $H$, see
\cite[Theorem~2]{BNS}, can be obtained as a direct consequence of
Theorem~\ref{intro2}.

\begin{corollary}\label{mp}
Let $H_1$ and $H_2$ be selfadjoint matrices in $\dC^n$
and assume that $H_1 \leq H_2$. Then
\[
 H_2^+ \leq H_1^+ \quad \textrm{if and only if} \quad
 \sfi(H_1) = \sfi(H_2).
 \]
\end{corollary}

It should be emphasized that both inequalities $H_2^{-1} \leq
H_1^{-1}$ and $-H_1^{-1} \leq -H_2^{-1}$ occur naturally in the
study of limits of monotone matrix functions, and they have
different geometrical implications; see \cite{BHSW2}. Such
inequalities between selfadjoint relations have interesting
applications, for instance, in the area of differential equations:
they appear in the study of the square-integrability of solutions of
definite canonical systems of differential equations; see
\cite{BHSW3} and the references therein.

The objective of this paper is to prove antitonicity results
analogous to Theorem~\ref{intro2} and Theorem~\ref{intro1}  for
selfadjoint operators or relations  $H_1$ and $H_2$  in a separable,
not necessarily finite-dimensional, Hilbert space. The results and
their proofs can be read with the finite-dimensional case in mind;
in fact, the proofs of the main two antitonicity theorems,
Theorem~\ref{finth2} and Theorem~\ref{antinew2} below, do not
essentially simplify in the finite-dimensional setting. As a
preparation some facts on selfadjoint relations in Hilbert spaces
are in given Section 2. In particular, the notion of ordering for
selfadjoint relations which are bounded from below and the concept
of inertia are introduced. Section~\ref{secanti} contains the main
results of the paper: the two infinite-dimensional variants of
Theorem~\ref{intro2} and \ref{intro1}. The important ingredients in
their proofs are an infinite-dimensional version of
Theorem~\ref{antith}, which has been independently established in
\cite{DM91b,S,HaNo0} (cf. \cite{HaNo}, and see also \cite{AK}), combined
with suitable perturbations arguments, and a general limit result on
monotone operator functions. Various consequences of the two main
antitonicity results are discussed, among them an
infinite-dimensional version of the antitonicity result for
Moore-Penrose inverses in Corollary~\ref{mp}.

\section{Ordering and inertia of selfadjoint relations}

This section contains an introduction to selfadjoint relations in
Hilbert spaces. In particular the notions of ordering and inertia
for selfadjoint relations in Hilbert spaces are introduced and
investigated.

\subsection{Linear relations}

Let $\sH$ be a Hilbert space with scalar product $(\cdot,\cdot)$
and corresponding norm $\|\cdot\|$. A \textit{(closed) relation} $H$ in $\sH$
is a (closed) linear subspace of the product space $\sH \times \sH$.
As such, $H$ is considered to consist of pairs
$\{h,k\} \in \sH \times \sH$, so that $H$ is the graph of
a multivalued (linear) operator in $\sH$.
The \textit{domain, range, kernel,} and \textit{multivalued part}
of a relation $H$ are defined as follows:
\[
\begin{split}
 \dom H&=\{h \in\sH :\,\{h,k\}\in H\}, \quad \ran H =\{k\in\sH :\,\{h,k\}\in H\},\\
  \quad \ker H&=\{h\in\sH :\,\{h,0\}\in H\}, \quad \mul H=\{k\in\sH :\,\{0,k\}\in H\}.
\end{split}
\]
Note that, if $H$ is closed then $\ker H$ and $\mul H$ are closed
subspaces. A number $\lambda \in \dC$ is called an
\textit{eigenvalue} of $H$ if $\{h, \lambda h\} \in H$ for some
nontrivial $h \in \sH$, which is then called an
\textit{eigenvector}. Similarly,  $\infty$ is said to be an
\textit{eigenvalue} of $H$ if $\{0,k\} \in H$ or, equivalently, $k
\in \mul H$, for some nontrivial  $k \in \sH$, which is then called
an  \textit{eigenvector}. The relation $H$ is an operator precisely
when $\mul H=\{0\}$, i.e., when $\infty$ is not an eigenvalue of
$H$.

Each relation $H$ has an \textit{inverse} $H^{-1}$ and an
\textit{adjoint} $H^*$, which are defined as
\begin{equation*}\label{definv}
\begin{split}
 H^{-1}&=\bigl\{\{k,h\}:\,\{h,k\}\in H\bigr\};\\
 H^*&=\bigl\{\,\{h,k\} \in \sH\times \sH :\,
 (g,h)=(f,k) \mbox{ for all } \{f,g\}\in H\,\bigr\}.
\end{split}
\end{equation*}
In particular, $\dom H^{-1}=\ran H$ and $\ker H^{-1}=\mul H$.
Note that the adjoint is a closed relation in $\sH$ and that it coincides with the usual adjoint when
$H$ is a densely defined operator.

For a relation $H$ in $\sH$ and $\lambda \in \dC$, the relation
$H-\lambda$ is given by
\[
 H-\lambda=\bigl\{\,\{h, k- \lambda h \}:\, \{h, k \} \in H\,\bigr\}.
\]
Its inverse, $(H-\lambda)^{-1}$, is a relation whose kernel and
multivalued part coincide with $\mul H$ and $\ker (H-\lambda)$,
respectively. Furthermore, it satisfies the following spectral
mapping identity:
\begin{equation}\label{Tra}
(H-\lambda)^{-1}
=-\frac{1}{\lambda}
+\frac{1}{\lambda^2}\left(-H^{-1}-\left(-\frac{1}{\lambda}\right)\right)^{-1},
\quad \lambda \in \dC \setminus \{0\}.
\end{equation}

For a closed relation $H$ the number $\lambda \in \dC$ is said to belong
to the \textit{resolvent set} of $H$, $\lambda \in \rho(H)$,
if $(H-\lambda)^{-1}$ is an everywhere defined operator.
The resolvent set is an open subset of $\dC$.
For $\lambda \in \rho(H)$ the operator
$(H-\lambda)^{-1}$ is called the \textit{resolvent operator}
of $H$ (at $\lambda$).

\subsection{Selfadjoint relations}
A relation $H$ is said to be \textit{symmetric} if $(k,h)\in\dR$ for
all $\{h,k\}\in H$. By the polarization formula, $H$ is symmetric
precisely when $H\subset H^*$. A relation $H$ is called
\textit{selfadjoint} if $H=H^*$; in particular, a selfadjoint
relation is automatically closed. A selfadjoint relation $H$ in
$\sH$ induces the following orthogonal decompositions of the space:
\begin{equation}\label{deck}
\sH = \cdom H \oplus \mul H \quad \textrm{and} \quad \sH = \cran H \oplus \ker H,
\end{equation}
where $\cdom H$ and $\cran H$ indicate the closures of
$\dom H$ and $\ran H$, respectively.
This shows that $H$ admits the following orthogonal decomposition:
\begin{equation}\label{null}
 H=H_s \hoplus (\{0\} \times \mul H),
\end{equation}
where $H_s=H \cap (\cdom H \times \cdom H)$, the so-called
\textit{orthogonal operator part} of $H$, is a selfadjoint operator
in $\cdom H$ and $\{0\} \times \mul H$ is a selfadjoint relation in
$\mul H$. The symbol $\hoplus$ in \eqref{null} indicates the
orthogonality of the summands. It follows from \eqref{null} that the
finite spectra of $H$ and of $H_s$ coincide. Hence
$\cmr\subset\rho(H)$ and the selfadjoint operator part $H_s$ is
bounded if and only if $\dom H$ is closed. Moreover, if $\ran H$ is
closed, then there exists a reduced neighborhood of $0$ in $\dR$
which belongs to $\rho(H)$, and $0$ is at most an isolated
eigenvalue of $H$. Let $E_s(\cdot)$ be the spectral function of
$H_s$ in $\cdom H$. Define the spectral function $E(\cdot)$ for $H$
in $\sH$  by $E(t)=E_s(t) \oplus 0_{\mul H}$, $t \in \dR$, (cf.
\eqref{null})  so that
\begin{equation}\label{reso}
 (H-\lambda)^{-1}= \int_\dR \frac{1}{s-\lambda}\,dE(s),
 \quad \lambda \in \rho(H).
\end{equation}
For a measurable function $\varphi: \dR \to \dC$, define
$\varphi(H)=\varphi(H_s)\hoplus (\{0\} \times \mul H)$.

A selfadjoint relation $H$ in a Hilbert space $\sH$ is said to be
\textit{bounded from below} by  $m \in \dR$ if its operator part
$H_s$ is bounded from below by $m$:
\[
(H_s h,h) \geq m (h,h) \quad \text{for all} \quad h \in \dom H=\dom H_s.
\]
Any such number $m$ is said to be \textit{a lower bound}.
The supremum of all lower bounds
is called \textit{the lower bound} of $H$. Any real
number smaller than the lower bound belongs to $\rho(H)$. If the
lower bound is nonnegative, then $H$ is called \textit{nonnegative}:
$H \geq 0$. Note that if $H$ has lower bound $m$, then $H-x$ has
lower bound $m-x$ for any $x \in \dR$. Therefore $H-x$ is
nonnegative for all $x \leq m$. In particular, if $x<m$ then
$(H-x)^{-1}$ is an everywhere defined positive bounded operator.

The \textit{square root} $H^{1/2}$
of a nonnegative selfadjoint relation $H$ is defined as
\begin{equation*}\label{nullh}
 H^{1/2}=(H_s)^{1/2} \hoplus (\{0\} \times \mul H).
\end{equation*}
For a nonnegative selfadjoint relation $H$ one has
\begin{equation}\label{nullh+}
\dom H \subset \dom H^{1/2}, \quad \cdom H=\cdom H^{1/2}, \quad \mul H=\mul H^{1/2}.
\end{equation}
Clearly, if $H_s$ is bounded, then
$\dom H=\dom H^{1/2}=(\mul H)^\perp$.

\subsection{Ordering of selfadjoint relations}
Let $H_1$ and $H_2$ be selfadjoint relations in a Hilbert space $\sH$
with lower bounds $m_1$ and $m_2$, respectively.
Then  $H_1$ and $H_2$  are said to satisfy $H_1\leq H_2$
if for a fixed $x<\min\{m_1,m_2\}$
\begin{equation}\label{ord01h}
0 \leq ((H_2-x)^{-1}h,h) \le ((H_1-x)^{-1}h,h)\quad \text{for all} \quad h \in \sH,
\end{equation}
see \cite{BHSW1,CS,HSSW}. The next proposition gives a characterization for the
ordering of selfadjoint relations, see \cite{CS,HSSW}.  According to this proposition
\eqref{ord01h} holds automatically for all $x<\min\{m_1,m_2\}$
if it holds for some $x<\min\{m_1,m_2\}$.

\begin{proposition}\label{oldh}
Let $H_1$ and $H_2$ be selfadjoint relations in a Hilbert space $\sH$
with lower bounds $m_1$ and $m_2$, respectively. Then
$H_1$ and $H_2$ satisfy $H_1 \leq H_2$ if and only if
for any $x< \min\{m_1,m_2\}$
\begin{equation}\label{dom}
\dom (H_{2}-x)^{1/2} \subset \dom (H_{1}-x)^{1/2}
\end{equation}
and
\begin{equation}\label{dom+}
\|(H_{1}-x)_s^{1/2}h\| \le \|(H_{2}-x)_s^{1/2}h\| \,\,\, \textrm{for all} \,\,\,  h \in
\dom (H_{2}-x)^{1/2}.
\end{equation}
\end{proposition}
If $\dom H_1$ and $\dom H_2$ are closed or, equivalently, if the
operator parts  $(H_{1})_s$ and $(H_{2})_s$ are bounded, then by
Proposition~\ref{oldh} (cf. \eqref{nullh+}) $H_1 \leq H_2$ if and
only if
\begin{equation}\label{ord00}
\dom H_{2} \subset \dom H_{1} \quad \text{and}\quad
 ((H_{1})_s h,h) \le ((H_{2})_s h,h) \quad \text{for all} \quad
h \in \dom H_{2}.
\end{equation}
In particular, if $\dom H_1=\dom H_2=\sH$, i.e., if
$H_1$ and $H_2$ are bounded selfadjoint operators, then the
inequality $H_1 \le H_2$ has the usual meaning.

The inclusion \eqref{dom} combined with \eqref{deck} and \eqref{nullh+}
yields the following implication
\begin{equation}\label{incl}
 H_1 \leq H_2 \quad \Rightarrow
 \quad \cdom H_2\subset\cdom H_1
 \quad\text{and}\quad\mul H_1 \subset \mul H_2.
\end{equation}

\begin{corollary}\label{inequh}
Let $H_1$ and $H_2$ be selfadjoint relations in a Hilbert space
$\sH$ with closed domains such that $H_1 \leq H_2$. Then $-H_2 \le
-H_1$ if and only if $\dom H_1 = \dom H_2$ or, equivalently, $\mul
H_1 = \mul H_2$.
\end{corollary}
\begin{proof}
By assumption the operator parts $(H_{1})_s$ and $(H_{2})_s$ are
bounded, which guarantees that each of the relations $\pm H_1$ and
$\pm H_2$ is bounded from below. Now, the implication
($\Rightarrow$) is obtained by applying \eqref{incl} to the
inequalities $H_1 \leq H_2$ and $-H_2 \le -H_1$. The implication
($\Leftarrow$) follows directly from \eqref{ord00}.
\end{proof}

Let $H_j$ be a selfadjoint relation in a Hilbert space $\sH$ with
lower bound $m_j$ and let $E_j(\cdot)$ be its spectral function for
$j=1,2$. Then for $x<m_j$,
\[
x\|h\|^2 + \|(H_j-x)^{1/2}_sh\|^2 = \int_\dR s \,d(E_j(s)h,h),
\quad h \in \dom (H_j-x)^{1/2}.
\]
Hence, the selfadjoint relations $H_1$ and $H_2$ satisfy  $H_1 \le
H_2$ if and only if the inclusion \eqref{dom} and the following
inequality are satisfied for any $x< \min\{m_1,m_2\}$:
\begin{equation}\label{compspms+}
\int_\dR s\,d(E_1(s)h,h) \leq \int_\dR s \,d(E_2(s)h,h) \quad \mbox{for all} \quad
h \in \dom (H_2-x)^{1/2}.
\end{equation}

The next lemma will be useful in the proofs of Proposition~\ref{ineqprop2}
and \ref{toegevoegd2} below.

\begin{lemma}\label{hulp}
Let $H_1$ and $H_2$ be selfadjoint relations in a Hilbert space
$\sH$ which are bounded from below and satisfy $H_1\le H_2$. Let
$E_1(\cdot)$ and $E_2(\cdot)$ denote the corresponding spectral
measures. Then the following statements hold:
\begin{itemize}
\item [{\rm (i)}] $\ran E_2((-\infty,0])\cap \ran(I-E_1((-\infty,0)))
\,\subset\, \ker H_1 \cap \ker H_2$;
\item [{\rm (ii)}] $\ran E_2((-\infty,0])\cap \ran(I-E_1((-\infty,0]))=\{0\}$;
\item [{\rm (iii)}] $\ran E_2((-\infty,0))\cap \ran(I-E_1((-\infty,0)))=\{0\}$.
\end{itemize}
\end{lemma}

\begin{proof}
Note first that since $H_2$ is semibounded,
\[\ran E_2((-\infty,0]) \subset \dom H_2\subset \dom(H_2-x)^{1/2}, \quad x< \min \{m_1,m_2\}.\]
Hence \eqref{compspms+} holds for $h \in \ran E_2((-\infty,0])$.

(i) For $h\in \ran E_2((-\infty,0])\cap \ran (I-E_1((-\infty,0)))$
the righthand side of \eqref{compspms+} is nonpositive
and the lefthand side is nonnegative. Hence,
\begin{equation*}
 \int_\dR s \,d(E_1(s)h,h) = \int_\dR s \,d(E_2(s)h,h)=0
\end{equation*}
and this implies (i).

(ii) Let $h\in \ran E_2((-\infty,0])\cap \ran (I-E_1((-\infty,0]))$. If $h\not= 0$,
the righthand side of \eqref{compspms+} is nonpositive
and the lefthand side is positive. Hence $h=0$ and (ii) holds.

(iii) Let $h\in \ran E_2((-\infty,0))\cap \ran (I-E_1((-\infty,0)))$. If  $h\not=0$,
the righthand side of \eqref{compspms+} is negative
and the lefthand side is nonnegative. Hence $h=0$ and (iii) holds.
\end{proof}

The following result is included as a preparation for Section~\ref{secanti}.

\begin{lemma}\label{s3}
Let $H$ be a selfadjoint relation in a Hilbert space $\sH$ and let
$(\alpha,\beta)$ be a spectral gap of $H$. Then the following
statements hold:
\begin{enumerate}\def\labelenumi {\rm (\roman{enumi})}
\item the relations $(H-\alpha)^{-1}$ and $(H-\beta)^{-1}$ are
selfadjoint with $-(H-\alpha)^{-1}$ and $(H-\beta)^{-1}$ being
bounded from below. They are limits
of $(H-t)^{-1}$ as $t \downarrow \alpha$
and $t \uparrow \beta$ respectively:
\begin{equation*}\label{dd0}
 (H-t)^{-1} \to (H-\alpha)^{-1}, \quad
 (H-t)^{-1} \to (H-\beta)^{-1},
\end{equation*}
where the convergence is in the strong resolvent sense.
Moreover, the inequalities
\begin{equation*}\label{dd1}
- (H-t)^{-1} \le -(H-\alpha)^{-1}\quad\text{and}\quad
 (H-t)^{-1} \le (H-\beta)^{-1}
\end{equation*}
hold  for $\alpha < t < \beta$;

\item if $K_\alpha$ and $K_\beta$ are selfadjoint relations in $\sH$
with $-K_\alpha$ and $K_\beta$  being bounded from below,  such that
$- (H-t)^{-1} \le -K_\alpha$ or $(H-t)^{-1} \le K_\beta$, $\alpha < t < \beta$,
then the limits $(H-\alpha)^{-1}$ and $(H-\beta)^{-1}$ satisfy
\begin{equation*}
-(H-\alpha)^{-1}\leq -K_\alpha\quad\text{or}
\quad (H-\beta)^{-1}\leq K_\beta.
\end{equation*}
\end{enumerate}
\end{lemma}

\begin{proof}
The statements are proved for the right endpoint $\beta$;
a similar reasoning applies to the left endpoint $\alpha$. Note first that
if $E(\cdot)$ is the spectral function of $H$, then \eqref{reso} shows that for all
$t_1,t_2\in(\alpha,\beta)$ with $t_1\le t_2$ and all $h\in \sH$,
\[
 ((H-t_2)^{-1}h,h)-((H-t_1)^{-1}h,h)
 =\int_{\dR\setminus (\alpha,\beta)} \frac{t_2-t_1}{(s-t_1)(s-t_2)}\, d(E(s)h,h).
\]
The support of the measure $d(E(\cdot)h,h)$ is contained in
$\dR\setminus (\alpha,\beta)$ and there the integrand is nonnegative.
Hence, the operator function $(H-t)^{-1}$
is nondecreasing in $t \in (\alpha, \beta)$.

(i) Fix some $c\in(\alpha,\beta)$ and let $m_c$ be a lower bound
for the bounded operator $(H-c)^{-1}$. As the function
$t\mapsto (H-t)^{-1}$ is nondecreasing in $(\alpha,\beta)$
it follows that $m_c$ is a lower bound for $(H-t)^{-1}$, $t\in(c,\beta)$.
Hence by \cite[Theorem 3.5]{BHSW1} there exists
a selfadjoint relation $B$ in $\sH$, bounded
from below by $m_c$, such that
$(H-t)^{-1}\rightarrow B$ as $t\uparrow\beta$
in the strong resolvent sense, or, equivalently,
in the graph sense; cf. \cite[Proposition 2.3]{BHSW1} and \cite{RS1}.
Moreover, $(H-t)^{-1}\leq B$ holds for all $t\in(c,\beta)$.

Hence, to prove (i) it suffices to verify $B=(H-\beta)^{-1}$. For this let $\{\phi, \psi\} \in B$.
Since $B$ is the graph limit of $(H-t)^{-1}$ there exist
$\{\phi_t, \psi_t\} \in (H-t)^{-1}$ with $\{\phi_t,\psi_t\} \to
\{\phi,\psi\}$ as $t \uparrow \beta$.  Since
\[
\{\psi_t, \phi_t+(t-\beta) \psi_t\} \in H-\beta \quad \mbox{and} \quad
\{\phi_t+(t-\beta) \psi_t, \psi_t\} \in (H-\beta)^{-1},
\]
it follows that $\{\phi, \psi\} \in (H-\beta)^{-1}$, i.e., $B\subset
(H-\beta)^{-1}$. Since both $B$ and $(H-\beta)^{-1}$  are
selfadjoint, the equality $B=(H-\beta)^{-1}$ follows.

(ii) Since $B=(H-\beta)^{-1}$ is bounded from below by $m_c$ the relation
$(H-\beta)^{-1}-m_c$ is nonnegative.
Recall that $\dom ((H-\beta)^{-1} -m_c)^{1/2}= \sH_0$, where
\[
\sH_0=\{ h \in \sH : \lim_{t \uparrow \beta} \|((H-t)^{-1} -m_c) ^{1/2}h\|< \infty\};
\]
cf.  \cite[Theorem 3.5]{BHSW1}. Now let $K_\beta$ be such that $(H-t)^{-1} \le K_\beta$, then by Proposition~\ref{oldh}
for all $t \in (c, \beta)$
\[
\dom (K_\beta-m_c)^{1/2} \subset \dom ((H-t)^{-1} -m_c) ^{1/2}
\]
and
\[
\|((H-t)^{-1} -m_c) ^{1/2}h\| \le \|(K_\beta -m_c)_s^{1/2}h\| \quad \textrm{for all}
\quad  h \in \dom (K_\beta-m_c)^{1/2}.
\]
Since $(H-t)^{-1}$ is a nondecreasing operator function on
$(\alpha,\beta)$, the preceding inequality implies that
\[
\dom (K_\beta -m_c)^{1/2} \subset \sH_0=\dom ((H-\beta)^{-1} -m_c)^{1/2}.
\]
Hence Proposition~\ref{oldh}
yields $(H-\beta)^{-1} \le K_\beta$.
\end{proof}

\subsection{Inertia of selfadjoint relations}

The notion of inertia of a selfadjoint relation in a Hilbert space
is defined by means of its associated spectral measure. In what
follows the Hilbert space is assumed to be separable.

\begin{definition}
Let $H$ be a selfadjoint relation in a separable Hilbert space $\sH$
and let $E(\cdot)$ be the spectral measure of $H$. The inertia of
$H$ is defined as the ordered quadruplet
$\sfi(H)=\bigl\{\sfi^+(H),\sfi^-(H),\sfi^0(H),\sfi^\infty(H)\bigr\}$,
where
\begin{equation*}
\begin{split}
\sfi^+(H)&=\dim \ran E((0,\infty)),\qquad \sfi^-(H)=\dim \ran
E((-\infty,0)),\\
\sfi^0(H)&=\dim\ker H,\qquad\qquad\quad\, \sfi^\infty(H)=\dim \mul H.
\end{split}
\end{equation*}
\end{definition}
In particular, for a selfadjoint relation $H$ in $\dC^n$, the
quadruplet $\sfi(H)$ consists of the numbers of positive, negative,
zero, and infinite eigenvalues of $H$; cf. \cite{BHSW2}. Hence, if
$H$ is a selfadjoint matrix in $\dC^n$, then ${\sf i}^\infty(H)=0$
and the remaining numbers make up the usual inertia of $H$, see,
e.g. \cite{HJ,LT} or the introduction.

The inertia numbers of a selfadjoint relation $H$ in a separable Hilbert space
$\sH$ satisfy:
\begin{equation}\label{nn}
\sfi^+(H)+\sfi^-(H)+\sfi^0(H)+\sfi^\infty(H)= \dim \sH.
\end{equation}
Furthermore, the following identities hold:
\begin{equation}\label{coninnum1}
\begin{split}
\sfi(H^{-1})&=\bigl\{\sfi^+(H),\sfi^-(H),\sfi^\infty(H),\sfi^0(H)\bigr\};\\
\sfi(-H^{-1})&=\bigl\{\sfi^-(H),\sfi^+(H),\sfi^\infty(H),\sfi^0(H)\bigr\}.\
\end{split}
\end{equation}

The next proposition shows that the ordering of two selfadjoint
relations in a separable Hilbert space implies certain inequalities
between their inertia numbers; cf. \cite[Proposition 3.6]{BHSW2} for
a finite-dimensional variant of Proposition~\ref{ineqprop2}.

\begin{proposition}\label{ineqprop2}
Let $H_1$ and $H_2$ be selfadjoint relations in a separable Hilbert
space $\sH$ which are bounded from below and satisfy $H_1\le
H_2$. Then their inertia
$\sfi(H_j)=\{\sfi^+_j,\sfi^-_j,\sfi^0_j,\sfi^\infty_j\}$, $j=1,2$,
satisfy the following inequalities:
\begin{itemize}
 \item [{\rm (i)}] ${\sf i}_1^\infty  \leq {\sf i}_2^\infty$ and ${\sf
i}_1^-+{\sf i}_1^0+{\sf i}_1^+  \geq {\sf i}_2^- + {\sf i}_2^0+ {\sf
i}_2^+$;
 \item [{\rm (ii)}] ${\sf i}_1^-  \geq {\sf i}_2^-$ and  ${\sf i}_1^-
+{\sf i}_1^0   \geq {\sf i}_2^-+{\sf i}_2^0$;
 \item [{\rm (iii)}] ${\sf i}_1^+ +{\sf i}_1^\infty   \leq {\sf i}_2^+
+{\sf i}_2^\infty$ and ${\sf i}_1^0 +{\sf i}_1^+ +{\sf i}_1^\infty
\leq {\sf i}_2^0 + {\sf i}_2^+ + {\sf i}_2^\infty$.
\end{itemize}
\end{proposition}

\begin{proof}
(i) This is a direct consequence of the implication in \eqref{incl}.

(ii) If $\textsf{i}_1^-=\infty$, then automatically ${\sf i}_2^-  \leq {\sf i}_1^-$.
Hence, in order to show  ${\sf i}_2^- \leq {\sf i}_1^-$,
assume that $\textsf{i}_1^-< \infty$ and
let $\cL$ be a finite-dimensional subspace in $\ran E_2((-\infty,0))$.
Since $E_1((-\infty,0))$ restricted to $\cL$ is injective by Lemma~\ref{hulp}~(iii),
one has
\begin{equation*}
\dim \cL = \dim E_1((-\infty,0))\sL \leq \dim\ran
E_1((-\infty,0))=\textsf{i}_1^-.
\end{equation*}
Thus any finite-dimensional subspace of $\ran E_2((-\infty,0))$
has dimension at most $\textsf{i}_1^-$, which implies that the space
$\ran E_2((-\infty,0))$ itself has dimension at most $\textsf{i}_1^-$, i.e.
${\sf i}_2^-  \leq {\sf i}_1^-$.

The inequality ${\sf i}_1^- +{\sf i}_1^0   \geq {\sf i}_2^-+{\sf i}_2^0$
can be shown in a similar way, when (ii) in Lemma~\ref{hulp} is used
instead of (iii).

(iii)  By Lemma~\ref{hulp}~(ii) the identity
\begin{equation}\label{empty2}
\ran E_2((-\infty,0])\cap (\ran E_1((0,\infty)) \oplus \mul H_1)=\{0\}
\end{equation}
holds.
If ${\sf i}_2^+ +{\sf i}_2^\infty=\infty$, then automatically
${\sf i}_1^+ +{\sf i}_1^\infty   \leq {\sf i}_2^+ +{\sf i}_2^\infty$.
Hence, in order to show ${\sf i}_1^+ +{\sf i}_1^\infty   \leq {\sf i}_2^+ +{\sf i}_2^\infty$,
assume that ${\sf i}_2^+ +{\sf i}_2^\infty < \infty$ and let $\cL$ be a finite-dimensional subspace in
$\ran E_1((0,\infty)) \oplus \mul H_1$.
Since $I-E_2((-\infty,0])$ restricted to $\sL$ is injective by \eqref{empty2}, one has
\begin{equation*}
\dim \cL = \dim \bigl(I-E_2((-\infty,0])\bigr)\sL \le \dim \ran \bigl(I-E_2((-\infty,0])\bigr)
={\sf i}_2^+ +{\sf i}_2^\infty.
\end{equation*}
Thus any finite-dimensional subspace of $\ran E_1((0,\infty)) \oplus \mul H_1$
has dimension at most ${\sf i}_2^+ +{\sf i}_2^\infty$, which implies that the space
$\ran E_1((0,\infty)) \oplus \mul H_1$ itself has dimension at most ${\sf i}_2^+ +{\sf i}_2^\infty$,
i.e.,  ${\sf i}_1^+ +{\sf i}_1^\infty   \leq {\sf i}_2^+ +{\sf i}_2^\infty$.

The inequality ${\sf i}_1^0 +{\sf i}_1^+ +{\sf i}_1^\infty \leq {\sf i}_2^- +{\sf i}_2^+ +{\sf
i}_2^\infty$
can be shown in a similar way, when (iii) in Lemma~\ref{hulp} is used
instead of (ii).
\end{proof}

The case of equality in an inertia inequality of Proposition~\ref{ineqprop2}
has a specific geometric implication.

\begin{proposition}\label{toegevoegd2}
Let $H_1$ and $H_2$ be selfadjoint relations in a separable
Hilbert space $\sH$ which are bounded from below. Let
$\sfi(H_j)=\{\sfi^+_j,\sfi^-_j,\sfi^0_j,\sfi^\infty_j\}$ be the
inertia of $H_j$, $j=1,2$, and assume that $H_1\le H_2$. Then
the following statements hold:

\begin{enumerate}\def\labelenumi {\rm (\roman{enumi})}
\item if ${\sf i}_1^\infty = {\sf i}_2^\infty< \infty$, then $\mul H_1 =
\mul H_2$;
\item if ${\sf i}_1^- +{\sf i}_1^0 = {\sf i}_2^-+{\sf i}_2^0<\infty$,
then $\ker H_1 \subset \ker H_2$;
\item if ${\sf i}_1^- = {\sf i}_2^-<\infty$, then $\ker H_2 \subset \ker
H_1$.
\end{enumerate}
In particular, if ${\sf i}_1^- = {\sf i}_2^-<\infty$ and ${\sf i}_1^0 =
{\sf i}_2^0<\infty$, then $\ker H_1 = \ker H_2$.
\end{proposition}

\begin{proof}
(i) This is a direct consequence of \eqref{incl}.

(ii) $\&$ (iii) Define the subspace $\sL_0=\ran E_2((-\infty,0])\cap
\ran (I-E_1((-\infty,0)))$. According to Lemma~\ref{hulp}~(i) $\sL_0
\subset \ker H_1\cap\ker H_2$. Furthermore, note that $\sL_0$ can be
rewritten as
\begin{equation*}\label{orth}
\sL_0=\ran E_2((-\infty,0])\cap \bigl(\ran E_1((-\infty,0))\bigr)^\perp.
\end{equation*}
Since $\dim \ran E_2((-\infty,0])=\textsf{i}_2^-+\textsf{i}_2^0$ and
$\dim \ran E_1((-\infty,0))=\textsf{i}_1^-<\infty$,
\begin{equation}\label{dimrr}
 \dim \sL_0   \ge \textsf{i}_2^-+\textsf{i}_2^0-\textsf{i}_1^-.
\end{equation}

In case (ii), the assumption together with \eqref{dimrr} implies that $\dim \sL_0 \geq \textsf{i}_1^0 = \dim \ker H_1$. Combining this observation
with the inclusion $\sL_0 \subset \ker H_1\cap\ker H_2 \subset \ker H_1$ yields that $\ker H_1 \cap \ker H_2=\ker H_1$
and, hence, that $\ker H_1 \subset \ker H_2$.

In case (iii), the assumption together with \eqref{dimrr} implies that $\dim \sL_0 \geq \textsf{i}_2^0 = \dim \ker H_2$. Combining this observation
with the inclusion $\sL_0 \subset \ker H_1\cap\ker H_2 \subset \ker H_2$ yields that $\ker H_1 \cap \ker H_2=\ker H_2$
and, hence, that $\ker H_2 \subset \ker H_1$.
\end{proof}

\section{Antitonicity for selfadjoint relations}\label{secanti}

The infinite-dimensional variants of the antitonicity theorems from the introduction are here proved
by means of perturbation arguments, the spectral mapping result \eqref{Tra}, and  limit properties of monotone operator functions.
Furthermore, various consequences and special cases of these results are also discussed.

\subsection{An antitonicity theorem for bounded and boundedly invertible operators}

The following theorem is the infinite-dimensional variant of Theorem
\ref{antith} from the introduction; it was proved independently in
\cite{DM91b, S, HaNo0}; cf. \cite{HaNo}. A simple proof is included here;
it relies on the main arguments used in \cite{HaNo0,HaNo}.

\begin{theorem}\label{antith2}
Let $H_1$ and $H_2$ be bounded and boundedly invertible
selfadjoint operators in a separable Hilbert space $\sH$.
Let $\sfi(H_j)=\{\sfi^+_j,\sfi^-_j,\sfi^0_j,\sfi^\infty_j\}$ be
the inertia of $H_j$, $j=1,2$, and assume that
$\min\{\sfi^+_2,\sfi^-_1\}<\infty$ and that $H_1\leq H_2$.
Then
\[
H_2^{-1} \leq H_1^{-1} \quad \textrm{if and only if}
\quad {\sf i}(H_1) = {\sf i}(H_2).
\]
\end{theorem}

\begin{proof}
Observe that for the bounded and boundedly invertible
selfadjoint operators $H_1$ and $H_2$ one has
$\sfi^0_j=0=\sfi^\infty_j$, $j=1,2$. Hence ${\sf i}(H_1) = {\sf i}(H_2)$
is equivalent to  $\sfi^-_1=\sfi^-_2$ and $\sfi^+_1=\sfi^+_2$.
Furthermore, observe that $\sfi^-_1 < \infty$ implies that $\sfi^-_2 < \infty$
and that $\sfi^+_2 < \infty$ implies that $\sfi^+_1 < \infty$;
cf. Proposition~\ref{ineqprop2}.

$(\Rightarrow)$  In view of \eqref{coninnum1} the equalities
$\sfi^-_1=\sfi^-_2$ and $\sfi^+_1=\sfi^+_2$ follow by applying
Proposition~\ref{ineqprop2} to $H_1 \leq H_2$ and $H_2^{-1} \leq
H_1^{-1}$.

$(\Leftarrow$) Assume that  ${\sf i}(H_1) = {\sf i}(H_2)$, so that
$\sfi^-_1=\sfi^-_2$ and $\sfi^+_1=\sfi^+_2$.
The asserted implication will be shown in two steps.

First consider the case that $\sfi^-_1<\infty$. Then
$\sfi^-_2=\sfi^-_1 < \infty$. Now define the operator $J$ as
$I_{{\sf i}_1^+} \oplus  -I_{{\sf i}_1^-}$. Then a result of
G.~K\"othe, cf. \cite[Satz~1.2]{Kothe}, shows the existence of
bounded and boundedly invertible operators $V_1$ and $V_2$ such that
\begin{equation*}
H_1 = V_1^* J V_1 \quad \textrm{and} \quad H_2 = V_2^*J V_2.
\end{equation*}
By means of the above notation the inequality $H_1 \leq H_2$ can be
written as
\begin{equation}\label{yxc}
0 \leq J - U^* J U, \qquad U = V_1 V_2^{-1}.
\end{equation}
A simple calculation shows that
\begin{equation*}
\small{
\begin{pmatrix} I & 0 \\ JU^* & I\end{pmatrix}^* \begin{pmatrix} J-UJU^* &0 \\ 0 &J \end{pmatrix}\begin{pmatrix} I & 0 \\ JU^* & I\end{pmatrix}
=
\begin{pmatrix} I & JU \\ 0 & I \end{pmatrix}^* \begin{pmatrix} J & 0 \\ 0 &J-U^*JU \end{pmatrix}\begin{pmatrix} I & JU \\ 0 & I \end{pmatrix}.
}
\end{equation*}
Since congruence does not change the inertia of bounded operators, the
inertia of the diagonal matrices in the above equation coincide, i.e.,
\[
{\sfi}^-(J-UJU^*)+\sfi^-(J) = \sfi^-(J)+{\sf i}^-(J-U^*JU).
\]
As $\sfi^-(J)=\sfi_1^-<\infty$ and $\sfi^-(J-U^*JU)=0$ by \eqref{yxc} it follows
that $\sfi^-(J-UJU^*)=0$ and
hence $J-UJU^*$ is a nonnegative
operator. Using the definition of $U$, this yields
\[
H_2^{-1}=(V_2^* JV_2)^{-1} = V_2^{-1} JV_2^{-*}
\leq V_1^{-1}J V_1^{-*} = (V_1^*JV_1)^{-1}=H_1^{-1},
\]
which completes the proof in the case $\sfi^-_1<\infty$.

Next consider the case $\sfi^+_2<\infty$. Then it follows that
$\sfi^+_1=\sfi^+_2 < \infty$.  By \eqref{coninnum1} this implies
that $\sfi^-(-H_1)=\sfi^-(-H_2)< \infty$. Since $H_1 \le H_2$ is
equivalent to $-H_2 \leq -H_1$, the previous step shows that
$-H_1^{-1} \leq -H_2^{-1}$, which is equivalent to $H_2^{-1} \le
H_1^{-1}$; see Corollary~\ref{inequh}. This completes the proof of
Theorem~\ref{antith2}.
\end{proof}

\subsection{First main antitonicity theorem}\label{firstthm}

The following theorem is the infinite-dimensional version of
Theorem~\ref{intro1} from the introduction. Recall that for
selfadjoint relations $H_1$ and $H_2$ with closed ranges the
operator parts of $H_1^{-1}$ and $H_2^{-1}$ are bounded; in
particular, the relations $-H_1^{-1}$ and $-H_2^{-1}$ are bounded
from below.

\begin{theorem}\label{finth2}
Let $H_1$ and $H_2$ be selfadjoint relations in a separable
Hilbert space $\sH$ which are bounded from below and have closed
ranges.
Let $\sfi(H_j)=\{\sfi^+_j,\sfi^-_j,\sfi^0_j,\sfi^\infty_j\}$
be the inertia of $H_j$, $j=1,2$, and
assume that  $\sfi^-_1+\sfi^0_1<\infty$ and that $H_1 \leq H_2$. Then
\[
 -H_1^{-1} \leq -H_2^{-1} \quad \textrm{if and only if} \quad
 {\sf i}_1^- +{\sf i}_1^0 = {\sf i}_2^- +{\sf i}_2^0.
 \]
\end{theorem}

\begin{proof}
$(\Rightarrow)$ Apply Proposition~\ref{ineqprop2} and \eqref{coninnum1} to the
inequalities $H_1 \leq H_2$ and  $-H_1^{-1} \leq -H_2^{-1}$.  Then the inertia equality
${\sf i}_1^- +{\sf i}_1^0 = {\sf i}_2^- +{\sf i}_2^0$ follows.

$(\Leftarrow)$ Let $H_1 \leq H_2$ and assume that ${\sf i}_1^- +{\sf
i}_1^0 = {\sf i}_2^- +{\sf i}_2^0<\infty$ holds.
Since the ranges of $H_1$ and $H_2$ are closed,
there exists a constant $\delta>0$, such that
$(-\delta,\delta)\backslash\{0\}\subset \rho(H_j)$; i.e.
$H_j$ has  a spectral gap around $0$ and the point $0$
is possibly an isolated eigenvalue of finite multiplicity, $j=1,2$.
Define $\mu^+:=\min\bigl\{1,\delta\bigr\}$,
then Proposition~\ref{oldh} implies that the inequality
\begin{equation}\label{finth22}
H_1 - \epsilon_1 \leq H_2 - \epsilon_2,
\qquad 0<\epsilon_2\le \epsilon_1<\mu^+,
\end{equation}
holds. Clearly,  $H_j(\epsilon_j):=H_j - \epsilon_j$
is boundedly invertible and its inertia is
\begin{equation}\label{finth3}
\sfi\bigl(H_j(\epsilon_j)\bigr)
= \bigl\{\sfi^+_j(\epsilon_j),\sfi^-_j(\epsilon_j),\sfi^0_j(\epsilon_j), \sfi^\infty_j(\epsilon_j)\bigr\}
=\bigl\{\sfi_j^+,\sfi^-_j+\sfi^0_j,0,\sfi^\infty_j\bigr\},\quad j=1,2.
\end{equation}
Let $m_j$ be a lower bound for $H_j$, $j=1,2$.
Then $m_j-1<m_j-\epsilon_j$
is a lower bound for $H_j(\epsilon_j)$, $j=1,2$.
Hence $H_1(\epsilon_1) \leq  H_2(\epsilon_2)$
in \eqref{finth22} implies that
\[
 0 \leq ( H_2(\epsilon_2) - x)^{-1} \leq (H_1(\epsilon_1) - x)^{-1},
 \quad x<\min\{0,m_1-1,m_2-1\};
\]
cf. \eqref{ord01h}. Using \eqref{Tra},
this yields the inequality
\begin{equation}\label{eqeqeqeq}
 \left( -H_2(\epsilon_2)^{-1} + 1/x\right)^{-1}
 \leq \left( -H_1(\epsilon_1)^{-1} + 1/x\right)^{-1}.
\end{equation}
By \eqref{coninnum1} and \eqref{finth3} the inertia numbers
of $-H_j(\epsilon_j)^{-1}$, $j=1,2$, are given by
\begin{equation}\label{inin}
\sfi\bigl(-H_j(\epsilon_j)^{-1}\bigr)=\bigl\{\sfi^-_j(\epsilon_j),\sfi^+_j(\epsilon_j),\sfi^\infty_j(\epsilon_j),0\bigr\},\quad
j=1,2.
\end{equation}
Since $(0,-1/(m_j-1))\subset\rho(-H_j(\epsilon_j)^{-1})$
the operator $-H_j(\epsilon_j)^{-1} + 1/x$ is bounded
and boundedly invertible
for all $x<\min\{0,m_1-1,m_2-1\}$, $j=1,2$.
Hence \eqref{inin} and \eqref{finth3} imply that for $j=1,2$:
\begin{equation*}\label{innum}
\begin{split}
\sfi\bigl(-H_j(\epsilon_j)^{-1}+1/x\bigr)
&=\bigl\{\sfi^-_j(\epsilon_j),\sfi^+_j(\epsilon_j)+\sfi^\infty_j(\epsilon_j),0,0\bigr\} =
\bigl\{\sfi^-_j+ \sfi^0_j,\sfi^+_j+\sfi^\infty_j,0,0\bigr\}.
\end{split}
\end{equation*}
Since by assumption $\sfi_1^-+\sfi_1^0=\sfi_2^-+\sfi_2^0< \infty$,
Theorem~\ref{antith2} applied to \eqref{eqeqeqeq} yields
\[
-H_1(\epsilon_1)^{-1} + 1/x\leq  -H_2(\epsilon_2)^{-1} + 1/x,
\quad 0<\epsilon_2\le \epsilon_1<\mu^+
\]
or, equivalently,
\begin{equation}\label{finth6}
 -(H_1 - \epsilon_1)^{-1}  \leq  -(H_2 - \epsilon_2)^{-1}, \quad 0<\epsilon_2\le
\epsilon_1<\mu^+.
\end{equation}
Now letting subsequently $\epsilon_2\downarrow 0$ and
$\epsilon_1\downarrow 0$ in \eqref{finth6} in the strong resolvent sense
and using Lemma~\ref{s3} in each step, the inequality
$-H_1^{-1} \leq -H_2^{-1}$ is obtained.
\end{proof}

It is emphasized that the equivalence in Theorem~\ref{finth2}
is not true without the minus signs; see Corollary~\ref{inequh}.

\begin{corollary}\label{cor1}
Let $H_1$ and $H_2$ be selfadjoint relations in
a separable Hilbert space $\sH$
with closed domains and closed ranges.
Let $\sfi(H_j)=\{\sfi^+_j,\sfi^-_j,\sfi^0_j,\sfi^\infty_j\}$
be the inertia of $H_j$, $j=1,2$, and assume that
$\sfi^-_1+\sfi^0_1<\infty$, $\sfi^\infty_2<\infty$,
and that $H_1 \leq H_2$.
Then the following statements are equivalent:
\begin{enumerate}
\def\labelenumi{\rm (\roman{enumi})}
\item ${\sf i}(H_1)={\sf i}(H_2)$;
\item \begin{enumerate}
\item[(a)] $-H_1^{-1}\le -H_2^{-1}$;
\item[(b)] $\mul H_1 = \mul H_2$;
\item[(c)] $\ker H_1 = \ker H_2$;
\end{enumerate}
\item $-H_2\leq -H_1$, $-H_1^{-1}\le -H_2^{-1}$, and $H_2^{-1}\le H_1^{-1}$.
\end{enumerate}
\end{corollary}

\begin{proof}
(i) $\Rightarrow$ (ii) This follows from Theorem~\ref{finth2} and
Proposition~\ref{toegevoegd2}.

(ii) $\Rightarrow$ (iii) Apply Corollary~\ref{inequh}
to the inequalities $H_1\leq H_2$ and
$-H_1^{-1}\leq -H_2^{-1}$.  Then the desired inequalities follow.

(iii) $\Rightarrow$ (i)
If the stated inequalities hold, then by Corollary~\ref{inequh}
$\mul H_1 = \mul H_2$ and $\ker H_1 = \ker H_2$,
i.e. $\sfi_1^\infty = \sfi_2^\infty$ and
$\sfi_1^0 = \sfi_2^0$. Furthermore, the inequality
$-H_1^{-1}\le -H_2^{-1}$ implies that
$\sfi_1^- + \sfi_1^0 = \sfi_2^- + \sfi_2^0$.
Since $\sfi^-_j$, $\sfi^0_j$, and $\sfi^\infty_j$ are finite for $j=1,2$,
\eqref{nn} shows that (i) holds.
\end{proof}

\subsection{Second main antitonicity theorem}

The following theorem is the infinite-dimensional version of
Theorem~\ref{intro2} from the introduction.
It is emphasized that in contrast to Theorem~\ref{finth2} there is
no closed range assumption on the relations.
However, the conditions $H_1 \le H_2$
and $\sfi^-_1<\infty$ imply $\sfi^-_2<\infty$; hence $H_1^{-1}$
and $H_2^{-1}$ are both semibounded from below.

\begin{theorem}\label{antinew2}
Let $H_1$ and $H_2$ be selfadjoint relations in
a separable Hilbert space $\sH$ which are bounded from below.
Let $\sfi(H_j)=\{\sfi^+_j,\sfi^-_j,\sfi^0_j,\sfi^\infty_j\}$
be the inertia of $H_j$, $j=1,2$, and assume that $\sfi^-_1<\infty$
and that $H_1 \leq H_2$. Then
\[
 H_2^{-1} \leq H_1^{-1} \quad \textrm{if and only if} \quad
 {\sf i}_1^- = {\sf i}_2^-.
 \]
\end{theorem}

\begin{proof}
$(\Rightarrow)$
Apply Proposition~\ref{ineqprop2} and \eqref{coninnum1} to
$H_1 \leq H_2$ and $H_2^{-1} \leq H_1^{-1}$. Then
the inertia equality ${\sf i}_1^- = {\sf i}_2^-$ follows.

$(\Leftarrow)$  Let $H_1 \leq H_2$ and assume that
${\sf i}_1^- = {\sf i}_2^-<\infty$ holds. Then the negative spectrum
of $H_j$ consists of $0\leq \sfi^-_j<\infty$
eigenvalues (counting multiplicities), $j=1,2$.
Let $\mu_j^-$ be the largest negative
eigenvalue of $H_j$ if $0<\sfi^-_j$ and define
$$
\mu^-:=\begin{cases} \min\{1,-\mu_1^-,-\mu_2^-\}, & \sfi_1^-=\sfi_2^->0,\\
        1, & \sfi_1^-=\sfi_2^-=0.
       \end{cases}
$$
Then
\begin{equation}\label{ineqh1h2U}
H_1 + \epsilon_1 \le H_2 + \epsilon_2,
\quad 0<\epsilon_1\le \epsilon_2<\mu^-,
\end{equation}
where $H_j + \epsilon_j$ is boundedly invertible
and $\sfi(H_j+\epsilon_j)=\{\sfi_j^+
+\sfi_j^0,\sfi_j^-,0,\sfi_j^\infty\}$, $j=1,2$.
Since by assumption $\sfi_1^-=\sfi_2^-<\infty$, Theorem~\ref{finth2}
can be applied to \eqref{ineqh1h2U} yielding
\[
 -(H_1+\epsilon_1)^{-1} \leq -(H_2+\epsilon_2)^{-1},
 \quad 0<\epsilon_1\le \epsilon_2<\mu^-.
\]
Because $(H_j+\epsilon_j)^{-1}$, $j=1,2$, is a bounded operator,
this inequality can be rewritten as
\begin{equation}\label{finth08U}
 (H_2+\epsilon_2)^{-1} \leq (H_1+\epsilon_1)^{-1},
 \quad 0<\epsilon_1\le \epsilon_2<\mu^-.
\end{equation}
Now letting subsequently $\epsilon_1\downarrow 0$
and $\epsilon_2\downarrow 0$ in \eqref{finth08U}
in the strong resolvent sense
and using Lemma~\ref{s3} in each step
(which is possible since $(-\mu^-,0)\subset \rho(H_j)$, $j=1,2$),
the inequality $H_2^{-1} \leq H_1^{-1}$ is obtained.
\end{proof}

Theorem~\ref{antinew2} with $\sfi^-_1=0$
implies the following well-known result
for nonnegative selfadjoint operators and relations;
cf. \cite{AN,CS}.

\begin{corollary}\label{cor5}
Let $H_1$ and $H_2$ be selfadjoint relations in
a separable Hilbert space $\sH$. Then
\[
 0\le H_1\le H_2 \quad \textrm{if and only if} \quad 0\le H_2^{-1}\le H_1^{-1}.
\]
\end{corollary}

The following corollary for (not necessarily bounded) selfadjoint operators
extends Theorem~\ref{antith2}; cf.  \cite[Theorems 1,~2]{S},
\cite[Theorem~2]{HaNo0}, and \cite[Theorem~1.4]{HaNo}.

\begin{corollary}\label{antinewcor}
Let $H_1$ and $H_2$ be injective selfadjoint operators in
a separable Hilbert space $\sH$ and let
$\sfi(H_j)=\{\sfi^+_j,\sfi^-_j,\sfi^0_j,\sfi^\infty_j\}$ be the inertia of $H_j$,
$j=1,2$. Then the following statements hold:
\begin{itemize}
 \item [{\rm (i)}] if $H_1$ and $H_2$ are bounded from below,
 $\sfi^-_1<\infty$, and $H_1\leq H_2$, then
\[ H_2^{-1} \leq H_1^{-1} \quad \textrm{if and only if} \quad {\sf i}(H_1) = {\sf i}(H_2);\]
\item [{\rm (ii)}] if $-H_1$ and $-H_2$ are bounded from below, $\sfi^+_2<\infty$, and $-H_2\leq -H_1$, then
\[ -H_1^{-1} \leq -H_2^{-1} \quad \textrm{if and only if} \quad {\sf i}(H_1) = {\sf i}(H_2).\]
\end{itemize}
\end{corollary}

Combining Theorem~\ref{finth2}, Theorem~\ref{antinew2},
and Proposition~\ref{toegevoegd2} yields the following result.

\begin{corollary}\label{cor6}
Let $H_1$ and $H_2$ be selfadjoint operators in
a separable Hilbert space $\sH$
which are bounded from below and have closed ranges.
Let $\sfi(H_j)=\{\sfi^+_j,\sfi^-_j,\sfi^0_j,\sfi^\infty_j\}$
be the inertia of $H_j$, $j=1,2$, and assume that
$\sfi^-_1+\sfi^0_1<\infty$
and that $H_1 \leq H_2$. Then
\[
H_2^{-1}\le H_1^{-1},\,\,\,\, -H_1^{-1}\leq -H_2^{-1}
\quad \textrm{if and only if} \quad \sfi_1^-=\sfi_2^-,\,\,\,\,
\sfi_1^0=\sfi_2^0,
\]
in which case $\ker H_1=\ker H_2$.
\end{corollary}

\subsection{An antitonicity theorem for Moore-Penrose inverses}

The so-called Moore-Penrose inverse $H^+$
of a selfadjoint operator $H$ in a Hilbert space is defined as
$$H^+:=PH^{-1}P,$$
where $H^{-1}$ is the inverse of $H$ (in the sense of relations)
and $P$ denotes the orthogonal projection onto $\cran H$ in $\sH$.
It follows that
\begin{equation}\label{MPh}
 H^+=(H^{-1})_s\hoplus (\ker H\times\{0\})
\end{equation}
holds. Note that the assumption $\mul H=\{0\}$ implies
$\ker (H^{-1})_s=\ker H^{-1}=\{0\}$
and hence $\ker H^+=\ker H$ and $\sfi(H^+)=\sfi(H)$ hold.

The following theorem is the infinite-dimensional version of
Corollary \ref{mp} from the introduction.

\begin{theorem}\label{th4}
Let $H_1$ and $H_2$ be selfadjoint operators in a separable
Hilbert space $\sH$ which are bounded from below. Let
$\sfi(H_j)=\{\sfi^+_j,\sfi^-_j,\sfi^0_j,\sfi^\infty_j\}$ be the inertia of
$H_j$, $j=1,2$, and assume that $\sfi^-_1+\sfi^0_1<\infty$ and that $H_1 \leq
H_2$. Then
\[
 H_2^+ \leq H_1^+ \quad \textrm{if and only if} \quad
 \ker H_1=\ker H_2 \text{ and } \,\,\sfi(H_1) = \sfi(H_2).
 \]
\end{theorem}

\begin{proof}
$(\Rightarrow)$ Since $\ker H_j^+ = \ker H_j$, $j=1,2$, it follows
from Proposition~\ref{ineqprop2} and \eqref{coninnum1} that
$\sfi^-_1=\sfi^-_2<\infty$ and $\sfi^0_1=\sfi^0_2<\infty$ hold.
Since $\sfi^\infty_1=\sfi^\infty_2=0$ by assumption, \eqref{nn}
implies that $\sfi^+_1=\sfi^+_2$ and, therefore, $\sfi(H_1) =
\sfi(H_2)$. The assertion $\ker H_1=\ker H_2$ follows from
Proposition~\ref{toegevoegd2}.

$(\Leftarrow)$ The assumption $\sfi(H_1) = \sfi(H_2)$ together with
Theorem~\ref{antinew2} implies the inequalities $H_2^{-1}\leq H_1^{-1}$
and $(H_2^{-1})_s\leq (H_1^{-1})_s$. Therefore, as
$\ker H_1=\ker H_2$ it follows from \eqref{MPh} and
Proposition~\ref{oldh} that $H_2^+ \leq H_1^+$ holds.
\end{proof}

\end{document}